\documentclass[12pt, reqno]{amsart}
\usepackage{amsmath, amssymb}

\newtheorem{theorem}{Theorem}[section]
\newtheorem{proposition}[theorem]{Proposition}

\newtheorem{corollary}[theorem]{Corollary}

\newtheorem{remark}[theorem]{Remark}

\newcommand{\Z}{\ensuremath{\mathbb{Z}}}

\RequirePackage{ifthen}
\usepackage[all]{xy}

\begin{document}

\title[Restriction Map]{On the Restriction Map for Jacobi Forms}

\author{B. Ramakrishnan}
\author{K. D. Shankhadhar}
\address{Harish-Chandra Research Institute, 
Chhatnag Road, Jhunsi, Allahabad 211 019 (India)\\
{\sl Current Address}: [K. D. Shankhadhar] The Institute of Mathematical Sciences, CIT Campus, Taramani,
Chennai 600 113 (India)}
\email[B. Ramakrishnan]{ramki@hri.res.in}
\email[K. D. Shankhadhar]{karam@imsc.res.in, karamdeo@gmail.com}


\begin{abstract}
In this article we give a description of the kernel of the restriction map for 
Jacobi forms of index 2 and obtain the injectivity of $D_0\oplus D_2$ on the space of 
Jacobi forms of weight 2 and index 2. We also obtain certain generalization of these 
results on certain subspace of Jacobi forms of square-free index $m$.
\end{abstract}

\maketitle

\vspace{-.8cm}

\section{Introduction}

Let $D_0$ be the restriction  map from the space of Jacobi forms of weight $k$, index $m$ on the congruence subgroup 
$\Gamma_0(N)$ to that of elliptic modular forms of the same weight on $\Gamma_0(N)$ 
given by $\phi(\tau,z) \mapsto \phi(\tau,0)$. 
More generally, one obtains modular forms of weight $k+\nu$ from  Jacobi
forms of weight $k$ by using certain differential operators $D_{\nu}$. 
Then it is known that the direct sum ~$\oplus_{\nu =0}^{m} D_{2\nu}$~ is injective for $k$ even, but in general 
the restriction map $D_0$ is not 
injective (see \cite{AB99, AB2003, e-z} for details). However, J. Kramer \cite{kramer} and T. Arakawa and S. B\"ocherer 
\cite{AB2003} observed that when $k=2$ the situation may be different. In fact, when $m=1$, Arakawa and B\"ocherer 
\cite{AB99} provided two explicit descriptions of  ${\rm Ker}(D_0)$: one in terms of modular forms of weight $k-1$ 
and the other 
in terms of cusp forms of weight $k+2$ (by applying the differential operator $D_2$ on ${\rm Ker}(D_0)$). 
In a subsequent paper \cite{AB2003}, they proved that $D_0$ is injective in the case $k=2$, $m=1$ and gave some 
applications. 
In a private communication to the authors, Professor B\"ocherer informed that one of his students gave a precise 
description of the image of 
$D_0 \oplus D_2$ in terms of vanishing orders at the cusps ($k$ arbitrary, $m=1$). Based on this,
he conjectured that in the case $k=2$, one can remove one of the $D_{2\nu}$ from the direct sum $\oplus_{\nu =0}^{m} 
D_{2\nu}$ without affecting the injectivity. 

In this paper, we generalize the results of \cite{AB99} to higher index.  In \S 3, we consider the case $m=2$ 
and show that ${\rm Ker}(D_0)$ is 
isomorphic to the space of vector-valued modular forms of weight $k-1$ and $D_2({\rm Ker}(D_0))$ is isomorphic to 
a certain subspace of 
cusp forms of weight $k+2$  and these two spaces are related with each 
other by a simple isomorphism (Theorem 3.3).  In \S 3.3, we obtain the 
injectivity of $D_0\oplus D_2$ on $J_{2,2}(\Gamma_0(2N))$, where $N=2$ or an odd square-free positive integer 
(Theorem 3.4). 
This confirms the conjecture made by Professor B\"ocherer partially in the 
index $2$ case (i.e., we can omit the operator $D_4$). 
In \S 4, we consider a subspace of $J_{k,m}(\Gamma_0(mN),\chi)$, where $m$ is a square-free positive integer and 
$N$ is any positive integer 
and obtain results similar to \cite{AB99} and prove the injectivity of $D_0$ on this subspace when $k=2$, and $mN$ 
is square-free (Theorem 4.2 and 
Corollary 4.4). In \S 5, we make several remarks concerning the subspace studied in \S 4. 

\section{Preliminaries}\label{sect:1}

Let $k$, $m$ and $N$ be positive integers and $\chi$ be a Dirichlet character
modulo $N$. We denote the standard generators of the full modular group ~$SL_2(\Z)$~ by   
~$S=\begin{pmatrix}0&-1\\1&0
\end{pmatrix}$, $T=\begin{pmatrix}1&1\\0&1\end{pmatrix}$~ 
and let $I_2=\begin{pmatrix}1&0\\0&1\end{pmatrix}$ be the identity matrix.  
One has the relation ~$S^2=(ST)^3=-I_2$. Let ~$\mathbb{H}$~ denote the complex upper half-plane. 
Define the action of the full Jacobi group ~$SL_2(\mathbb{Z})\ltimes \mathbb{Z}^2$~ on 
~$\mathbb{H}\times \mathbb{C}$~ by
\begin{equation}\label{action}
[\gamma,(\lambda,\mu)](\tau,z)=\left(\frac{a\tau+b}{c\tau+d},\frac{z+\lambda\tau+\mu}{c\tau+d}\right),
\end{equation}
where ~$\gamma=\begin{pmatrix}a&b\\c&d\end{pmatrix} \in SL_2 (\mathbb{Z})$, 
~$\lambda,\mu \in \mathbb{Z}$ and $(\tau,z) \in \mathbb{H} \times \mathbb{C}$. 
We  simply write ~$\gamma(\tau,z)$~for ~\\
$[\gamma,(0,0)](\tau,z)$. We denote the space of Jacobi forms 
of weight $k$, index $m$
and Dirichlet character $\chi$ for the Jacobi group $\Gamma_0(N) \ltimes
\mathbb{Z}^2$ by $J_{k,m}(\Gamma_0(N),\chi)$.
It is well known that any such Jacobi form ~$\phi(\tau,z)$~
can be (uniquely) 
written as
\begin{equation*}\label{eq:1}
\phi(\tau,z) = \sum_{r=0}^{2m-1} h_{m,r}(\tau)\theta_{m,r}^J(\tau,z),
\end{equation*}
with
\begin{equation*}\label{eq:2}
\begin{split}
\theta_{m,r}^J(\tau,z) ~&=~ \sum_{n \in 
\mathbb{Z}}e^{2 \pi i m((n+\frac{r}{2m})^2\tau+2(n+\frac{r}{2m})z)},\\
h_{m,r}(\tau)& = \displaystyle\sum_{n \in \mathbb{Z}\atop{n\ge r^2/4m}}
c_{\phi}(n,r)e^{2\pi i(n-\frac {r^2}{4m})\tau},
\end{split}
\end{equation*}
where $c_\phi(n,r)$ denotes the $(n,r)$-th Fourier coefficient of the Jacobi 
form $\phi$.

\noindent The (column) vector $\Theta^J(\tau,z) = (\theta_{m,r}^J(\tau,z))_{0\leqslant r <2m}$ satisfies
the transformation 
\begin{equation}\label{eq:3}
\Theta^J(\gamma(\tau,z)) = e^{2\pi i m \frac{cz^2}{c\tau +d}}
(c\tau +d)^{\frac{1}{2}} U_m(\gamma) \Theta^J(\tau,z) 
\end{equation}
for all $\gamma = \begin{pmatrix}a&b\\c&d\end{pmatrix}\in SL_2
\mathbb{(Z)}$. Here ~$U_m : SL_2(\mathbb{Z}) \longrightarrow U(2m,\mathbb{C})$~
is a (projective) representation of $SL_2(\mathbb{Z})$. Note that ~$U_m$~ is an example of
a Weil representation (associated with the discriminant form $(D,Q)$ with
~$D=\Z/2m\Z$~ and ~$Q:D \rightarrow \mathbb{Q}/\Z$~ given by $Q(x)=x^2/4m$). 
In the cases $m = 1$ and $2$, it is given by
\begin{equation*}\label{eq:4}
U_1(T)~=~\begin{pmatrix}1&0\\0&
i\end{pmatrix}, ~
U_1(S)~=~\frac{e^{-\pi i/4}}
{\sqrt 2}\begin{pmatrix}1&1\\1&-1\end{pmatrix}.
\end{equation*}
\begin{equation*}\label{eq:5} 
U_2(T)~=~\begin{pmatrix}1&0&0&0
\\0&\sqrt{i}&0&0\\0&0&-1&0\\0&0&0&\sqrt{i}\end{pmatrix}, 
U_2(S)~=~\frac{e^{-\pi i/4}}{2}
\begin{pmatrix}1&1&1&1\\1&-i&-1&i\\1&-1&1&-1\\1&i&-1&-i\end{pmatrix}. 
\end{equation*}
The (column) vector ${\bf{h}} = (h_{m,r})_{0\leqslant{r}<{2m}}$ satisfies the 
following transformation property. For any 
$\gamma = \begin{pmatrix}a&b\\c&d\end{pmatrix} \in \Gamma_0(N)$, we have
\begin{equation}\label{eq:6}
{\bf{h}}(\gamma\tau) = \chi(d) (c\tau+d)^{k-\frac{1}{2}}{\overline{U_m}}(\gamma){\bf{h}}
{(\tau)}, 
\end{equation}
where ~$\gamma\tau = \frac{a\tau+b}{c\tau+d}$.

\noindent 
For any $\gamma = \begin{pmatrix}a&b\\c&d\end{pmatrix} \in \Gamma_0(m)$, 
let $\gamma_m$ denote the $SL_2(\mathbb{Z})$ 
matrix $\begin{pmatrix} a&bm \\ \frac{c}{m}&d
\end{pmatrix}$. Let the matrices $(u_{ij})_{0 \leqslant i,j <2m}$ and 
$(u_{ij}^m)_{0 \leqslant i,j \leqslant 1}$ represent $U_m(\gamma)$ and $U_1(\gamma_m)$ respectively.
Since $\theta_{m,0}^J(\tau,z)=\theta_{1,0}^J(m\tau,mz)$ and $\theta_{m,m}^J(\tau,z)=
\theta_{1,1}^J(m\tau,mz)$,  for any $\gamma = \begin{pmatrix}a&b\\c&d\end{pmatrix}\in \Gamma_0(m)$, 
we have
\begin{equation}\label{eq:7}
\begin{split}
(\theta_{m,0}^{J}, \theta_{m,m}^J)^t (\gamma(\tau,z)) ~&=~ (\theta_{1,0}^{J}, \theta_{1,1}^J)^t (m(\gamma(\tau,z)))
~=~ (\theta_{1,0}^{J}, \theta_{1,1}^J)^t (\gamma_m(m\tau,mz))\\
&= e^{2 \pi im{\frac{cz^2}{c\tau +d}}}(c\tau +d)^{\frac{1}{2}}U_1(\gamma_m) (\theta_{m,0}^{J}, \theta_{m,m}^J)^t (\tau,z).
\end{split}
\end{equation}  
Now, comparing the 
transformation properties for the action of $\gamma \in \Gamma_0(m)$ 
as given in \eqref{eq:3} and \eqref{eq:7} for 
$\theta_{m,0}^J$ and $\theta_{m,m}^J$, we have two linear equations in
$(\theta_{m,r}^J)_{0\leqslant r <2m}$ as follows:
\begin{equation}\label{eq:8}
(u_{00}-u_{00}^m)\theta_{m,0}^J+ (u_{0m}-u_{01}^m)
\theta_{m,m}^J+ \sum_{j\neq 0,m}u_{0j}\theta_{m,j}^J ~=~ 0.
\end{equation}
\begin{equation}\label{eq:9}
(u_{m0}-u_{10}^m)\theta_{m,0}^J+ (u_{mm}-u_{11}^m)
\theta_{m,m}^J+ \sum_{j\neq 0,m}u_{mj}\theta_{m,j}^J ~=~ 0.
\end{equation} 
Since the set $\{ \theta_{m,r}^J \}_{0\leqslant r <2m}$ is linearly independent
over the field of complex numbers $\mathbb{C}$ (see for example \cite[Lemma 3.1]{Z}), we have
\begin{equation*}
\begin{split}
u_{00} &=u_{00}^m, u_{0m}=u_{01}^m, u_{0j}=0 ~~\mbox{for all}~~ j\neq 0,m;\\
u_{m0}&=u_{10}^m, u_{mm}=u_{11}^m, u_{mj}=0 ~~\mbox{for all}~~ j\neq 0,m.
\end{split}
\end{equation*}
Using \eqref{eq:6} and the above observation for $U_m$, we also have
\begin{equation}\label{eq:10}
(h_{m,0}, h_{m,m})^t (\gamma\tau) ~=~ \chi(d) (c\tau+d)^{k-\frac{1}{2}}{\overline{U_{1}}(\gamma_m)} 
(h_{m,0}, h_{m,m})^t (\tau),
\end{equation}
for all $\gamma
 = \begin{pmatrix}a&b\\c&d\end{pmatrix} \in \Gamma_0(mN)$.\\
For any two matrices ~$\gamma, \gamma' \in \Gamma_0(m), (\gamma\gamma')_m=\gamma_m\gamma'_m$. 
Using this property we define a character $\omega_m$ on $\Gamma_0(m)$ by 
\begin{equation}\label{omega}
\omega_m(\gamma) ~=~ \det(U_1(\gamma_m)) \qquad (\gamma \in \Gamma_0(m)).
\end{equation}
Note that the character ~$\omega_1$~ is the same as the character ~$\omega$~ as defined in \cite[p. 311]{AB99}.

Set $\theta_{m,r}(\tau) := \theta^J_{m,r}(\tau,0)$. Putting ~$z=0$~ in \eqref{eq:3}, we get the following
transformation property for the column vector ~$\Theta=(\theta_{m,r})_{0\leq r <2m}$.
\begin{equation}\label{theta-trans}
\Theta(\gamma\tau) = 
(c\tau +d)^{1/2} U_m(\gamma) \Theta(\tau), 
\end{equation}
for any ~$\gamma=\begin{pmatrix}a&b\\c&d\end{pmatrix}\in SL_2(\Z)$. 
Let $\eta(\tau) = e^{2\pi i\tau/24}\prod_{n\ge 1}(1-e^{2\pi in\tau})$ denote the Dedekind eta function.
Combining the equations $(76.1),(78.4)$ and $(78.6)$ of \cite{rada}, one has the following identity:
\begin{equation}\label{eta-identity}
\eta^3(2\tau) = \frac{1}{2} \theta_{1,0}(\tau) \theta_{1,1}(\tau)\sum_{n\in {\mathbb Z}}(-1)^n e^{2\pi in^2\tau}.
\end{equation}
The above identity implies that $\theta_{1,0}$ and $\theta_{1,1}$ have no  
zeros in the upper half-plane.

\noindent
Let ~$\xi(\tau) = (\theta_{1,1}\theta_{1,0}' - \theta_{1,0}\theta_{1,1}')(\tau)$~ be the cusp form of weight $3$ for $SL_2({\mathbb Z})$ with character $\omega$ (see \cite[Proposition 2]{AB99}). 
For any ~$m\geq 1$, define 
\begin{equation}\label{xi_m}
\xi^*_m(\tau) := (\theta_{m,m}\theta_{m,0}' - \theta_{m,0}\theta_{m,m}')(\tau).
\end{equation}
Since ~$\theta_{m,0}(\tau)=\theta_{1,0}(m\tau)$~ and ~$\theta_{m,m}(\tau)=\theta_{1,1}(m\tau)$, we get that
~$\xi_m^*(\tau)=m\xi(m\tau)$. Therefore, $\xi_m^*$~ is a cusp form of weight $3$ for the group $\Gamma_0(m)$ 
with character $\omega_m$. Note that ~$\xi_1^*=\xi$.
It is proved in \cite[Proposition 2]{AB99} and \cite{kramer_1}
that ~$\xi(\tau)=-\pi i \eta^6(\tau)$. Therefore,  
\begin{equation}\label{xi}
\xi^*_m(\tau) ~=~ -\pi i m \eta^6(m\tau)
\end{equation}
and hence ~$\xi_m^*$~ has no zeros in ~$\mathbb{H}$.
 
Let $M_k(\Gamma_0(N),\chi)$ (resp. $S_k(\Gamma_0(N),\chi)$) denote the vector space of modular forms (resp.  cusp forms) 
of weight $k$ for the group $\Gamma_0(N)$ with character $\chi$. 
Let ~$D_0 : J_{k,m}(\Gamma_0(N),\chi) \longrightarrow 
M_k(\Gamma_0(N),\chi)$~ be the restriction map given by $\phi(\tau,z) 
\mapsto \phi(\tau,0)$ and  $D_2$ be the differential operator  
\begin{equation}\label{eq:15}
D_2 ~=~ {\left(\frac{k}{2\pi i}\frac{\partial^2}{\partial z^2} - 
4\frac{\partial}{\partial\tau}\right)} 
\Big|_{z=0}, 
\end{equation}
which acts on holomorphic functions on $\mathbb{H} \times \mathbb{C}$. 
Note that ~$D_2$~ maps the space Jacobi form $J_{k,m}(\Gamma_0(N),\chi)$ into the space of 
cusp forms $S_{k+2}(\Gamma_0(N), \chi)$. We denote the kernel of the restriction map ~$D_0:
J_{k,m}(\Gamma_0(N),\chi)\rightarrow M_{k}(\Gamma_0(N),\chi)$~ by ~$J_{k,m}(\Gamma_0(N),\chi)^0$.
Without loss of generality we assume that the space ~$M_k(\Gamma_0(N),\chi)$~ is non-empty, otherwise  
${\rm Ker}(D_0)$ is the entire space ~$J_{k,m}(\Gamma_0(N),\chi)$. In particular, this implies that $\chi(-1)=(-1)^k$. Using the transformation property of any Jacobi form for the matrix ~$-I_2$~, we get the symmetry relation 
~$h_{m,r}(\tau) = h_{m,2m-r}(\tau)$~ for all $r \in \Z/2m\Z$. 

\section{The space of Jacobi forms of index $2$}\label{sect:3}
 
Throughout this section, we assume that  $N= 2$, or  an odd square-free positive integer.
In this section, we study the kernel of the restriction map $D_0$ 
for the space of index $2$ Jacobi forms on $\Gamma_0(2N)$ and deduce the injectivity 
of $D_0\oplus D_2$ in the weight $2$ case. 

It is elementary to verify that the set 
$$
X = \left\{V = (v_{ij})_{0\leq i,j\leq 3} : v_{ij}=0 \iff i+j \not\equiv 0\pmod{2}, v_{11}=v_{33}, v_{13}=v_{31}\right\}
$$
is a subgroup of $GL_4({\mathbb C})$ and hence it is easy to verify that the function 
\begin{equation}\label{r-char}
r(V) = v_{11} + v_{13}
\end{equation}
is a character on the subgroup $X$. It is known that  the group $\Gamma_0(2)$ is generated by the matrices 
~$-I_2, T$~ and ~$ST^2S$. The action of $U_2$ on the generators of $\Gamma_0(2)$ is given by
\begin{equation*}\label{eq:18}
U_2(-I_2) = \begin{pmatrix}
-i&0&0&0\\0&0&0&-i\\0&0&-i&0\\0&-i&0&0\end{pmatrix}, 
U_2(T) = \begin{pmatrix} 1 &0&0&0\\ 0& \sqrt{i}&0&0\\  0&0&-1&0\\ 0&0&0&\sqrt{i}\end{pmatrix}  
\end{equation*}
and
\begin{equation*}
U_2(ST^{2}S) = \frac{1}{2i}\begin{pmatrix}
1+i&0&1-i&0\\0&1-i&0&1+i\\1-i&0&1+i&0\\0&1+i&0&1-i\end{pmatrix}.
\end{equation*}
Therefore, it follows that $U_2(\gamma) \in X$ for every $\gamma\in \Gamma_0(2)$. Thus, 
we have a (projective) representation of $\Gamma_0(2)$ defined by 
\begin{equation}\label{rho2}
\rho_2(\gamma) = r(U_2(\gamma))^{-1} \overline{U}_1(\gamma_2),
\end{equation}
where $\gamma_2=\begin{pmatrix}a&2b\\c/2&d\end{pmatrix}\in SL_2({\mathbb Z})$, with  
$\gamma=\begin{pmatrix}a&b\\c&d\end{pmatrix}\in \Gamma_0(2)$ and $r$ is the character defined 
by \eqref{r-char}. Since $U_2(\gamma) \in X$ for all 
$\gamma \in \Gamma_0(2)$, equation  
\eqref{theta-trans} gives us the  following transformation properties.
\begin{equation}\label{eq:3.1}
\theta_{2,1}(\gamma\tau)=(c\tau+d)^{1/2}(u_{11}\theta_{2,1}+u_{13}\theta_{2,3})(\tau),
\end{equation}
\begin{equation}\label{eq:3.2}
\theta_{2,3}(\gamma\tau)=(c\tau+d)^{1/2}(u_{13}\theta_{2,1}+u_{11}\theta_{2,3})(\tau),
\end{equation}
where $U_2(\gamma) = (u_{ij})_{0\le i,j\le 3}$ as assumed before. One also has the following relations.
\begin{equation}\label{theta-23}
\theta_{2,1}(\tau)=\sum_{n\in \Z}e^{4\pi i(n+\frac{1}{4})^2\tau}=\sum_{n\in \Z}e^{4\pi i(n-\frac{1}{4})^2\tau}=
\sum_{n\in \Z}e^{4\pi i(n+\frac{3}{4})^2\tau}=\theta_{2,3}(\tau),
\end{equation}
\begin{equation}\label{theta-12}
\begin{split}
\theta_{1,1}(\tau)=\sum_{n\in \Z}e^{2\pi i(n+\frac{1}{2})^2\tau}=\sum_{n\in \Z}e^{2\pi i(2n+\frac{1}{2})^2\tau}+
\sum_{n\in \Z}e^{2\pi i((2n+1)+\frac{1}{2})^2\tau}
& =\theta_{2,1}(2\tau)+\theta_{2,3}(2\tau)\\
& =2\theta_{2,1}(2\tau).
\end{split}
\end{equation}
From \eqref{theta-12}, we have ~$\theta_{2,1}(\tau)=\frac{1}{2}\theta_{1,1}(\tau/2)$~ and since ~$\theta_{1,1}$~ has no zeros in $\mathbb{H}$, $\theta_{2,1}$~ also has no zeros in the upper half-plane $\mathbb{H}$.

\subsection{Connection to the space of vector valued modular forms}
\label{subsect:4}
 
We start now from a Jacobi form $\phi \in J_{k,2}(\Gamma_0(2N), \chi)^0$, the kernel of the restriction 
map $D_0$. This implies that 
\begin{equation*}\label{eq:21}
0=\phi(\tau,0)=h_{2,0}(\tau)\theta_{2,0}(\tau)+2h_{2,1}
(\tau)\theta_{2,1}(\tau) +
h_{2,2}(\tau)\theta_{2,2}(\tau).
\end{equation*}
Define two new functions as 
\begin{equation}\label{eq:22}
\varphi_0(\tau) := {\frac{h_{2,0}}{\theta_{2,1}}}(\tau), \qquad  
\varphi_2(\tau) := {\frac{h_{2,2}}{\theta_{2,1}}}(\tau).
\end{equation}
As observed before, $\theta_{2,1}$ has no zeros in the upper half-plane, and hence, $\varphi_0$ and $\varphi_2$ are holomorphic in the upper half-plane.
\begin{proposition}\label{prop:2}
Let  $\varphi_0$ and $\varphi_2$ be as in \eqref{eq:22}. Then,  $(\varphi_0, \varphi_2)^t$ is a vector valued modular form on $\Gamma_0(2N)$ of weight $(k-1)$ with character $\chi$ and representation $\rho_2$.
We denote the space of all such vector valued modular forms by $M_{k-1}(\Gamma_0(2N),\chi;\rho_2)$.
\end{proposition}
\begin{proof} \label{proof:2} 
We have to check the transformation property and holomorphy 
condition at the cusps. First we check the transformation property. 
Let $\gamma = \begin{pmatrix}a&b\\c&d\end{pmatrix} \in \Gamma_0(2N)$ with $\gamma_2 = 
\begin{pmatrix} a& 2b \\ c/2& d\end{pmatrix}$ as defined in \S 2. Using the definition of 
$\varphi_0, \varphi_2$ and using the transformation \eqref{eq:10} with $m=2$, 
we have
\begin{equation}\label{eq:24}
\theta_{2,1}(\gamma\tau)(\varphi_0, \varphi_2)^t
(\gamma\tau) = \chi(d) (c\tau+d)^{k-\frac{1}{2}} {\overline{U_1}(\gamma_2)}\theta_{2,1}(\tau)
(\varphi_0, \varphi_2)^t(\tau).
\end{equation}
Using \eqref{eq:3.1} with \eqref{theta-23},  we get ~$\theta_{2,1}(\gamma\tau) = 
(u_{11}+u_{13})(c\tau+d)^{\frac{1}{2}}
\theta_{2,1}(\tau)$. Since, $(u_{11}+u_{13}) = r(U_2(\gamma))$, by using the definition of the representation $\rho_2$ given by 
\eqref{rho2}, the above formula \eqref{eq:24} gives the following transformation property:
\begin{equation}\label{eq:25}
(\varphi_0,\varphi_2)^t(\gamma\tau) = \chi(d)(c\tau+d)^{k-1} 
{\rho_2}(\gamma)(\varphi_0, \varphi_2)^t(\tau).
\end{equation}
It remains to investigate the behaviour of
$\varphi_0$ and $\varphi_2$ at each cusp of $\Gamma_0(2N)$. A complete 
set of cusps of $\Gamma_0(2N)$ is given by the numbers $\frac{a}{c}$ 
where $c$ runs over positive divisors of $2N$ and for a given $c$, 
$a$ runs through integers with $1 \leqslant a \leqslant 2N$, 
gcd$(a,2N) = 1$ that are inequivalent modulo $\gcd(c,\frac{2N}{c})$. 
For any $N$, the cusps corresponding to the divisors $1$ and $2N$, 
(i.e., the cusps $1$ and $\frac{1}{2N}$) are equivalent to $0$
and $\infty$ respectively. So we can assume 
that all the cusps of $\Gamma_0(2N)$ are given by $\infty$, 
$0$ and $\frac{1}{c}$ 
with $c|2N$ and $c \neq 1,2N$. Now choose any 
cusp $s$ from the above list. We divide our cusp condition 
verification into the following two cases :\\
\underline{Case 1:} ~Suppose that $s = \infty $. Let us denote ~$e^{2\pi i \tau}$~ by ~$q$.
If $\phi \in {\rm Ker}(D_0)$, then $\phi(\tau,0) = \sum c_\phi(n,r) q^n =0$ implies that $c_{\phi}(0,0) = 0$.
Therefore, $h_{2,0}(\tau) = \displaystyle\sum_{n\geqslant0}c_{\phi}(n,0)q^n 
= q\displaystyle\sum_{n\geqslant 1}c_{\phi}(n,0)q^{n-1}$. 
Also we have $h_{2,2}(\tau) = \displaystyle\sum_{n \geqslant 
\frac{1}{2}}c_{\phi}(n,2) q^{n-\frac{1}{2}} = 
q^{\frac{1}{2}}\displaystyle\sum_{ n\geqslant 1}c_{\phi}(n,2)
q^{n-1}$ 
and $\theta_{2,1}(\tau) = \displaystyle\sum_{n\in 
\mathbb{Z}} q^{2(n+\frac{1}{4})^2} = q^{\frac{1}{8}}\displaystyle
\sum_{n\in \mathbb{Z}}q^{(2n^2+n)}$.
Now the holomorphicity at the cusp $\infty$ of the functions 
$\varphi_0 = \frac{h_{2,0}}{\theta_{2,1}}$ and $\varphi_2 = \frac{h_{2,2}}{\theta_{2,1}}$ 
follow from the $q$ expansions of $h_{2,0}$, $h_{2,2}$ and $\theta_{2,1}$.\\
\underline{Case 2:} ~Suppose that $s \neq \infty$ and choose a matrix ~$g\in SL_2(\Z)$~ such that ~$g(\infty)=s$.  Explicitly, if $s=0$ or $s = \frac{1}{c}$ we choose $g$ as $S$ or $ST^{-c}S$. Then we have
\begin{equation*}
\begin{split}
(c\tau+d)^{-k+\frac{1}{2}}(-2 h_{2,1}(g \tau)) &= (c\tau +d)^{-k+
\frac{1}{2}}(\varphi_0(g\tau), \varphi_2(g\tau))(\theta_{2,0}(g\tau), 
\theta_{2,2}(g\tau))^t\\
&= (c\tau + d)^{-k+1}(\varphi_0(g\tau),\varphi_2(g\tau)) (c\tau+d)^{-\frac{1}{2}}
(\theta_{2,0}
(g\tau), \theta_{2,2}(g\tau))^t.
\end{split}
\end{equation*}
If ~$g=ST^{-c}S$~ then the ~$(0,0)$-th and ~$(2,0)$-th entries of ~$U_2(g)=U_2(S)U_2(T^{-c})U_2(S)$~ are equal (upto some constants) ~$1+(-1)^c+2i^{-c/2}$~ and ~$1+(-1)^c-2i^{-c/2}$~ respectively.
Thus, for the above choices of $g$, the matrix $U_2(g)$ will have nonzero entries at the $(0,0)$-th place and the $(2,0)$-th place since either $g$ is $S$ or $2\Vert c$ or $c$ is odd. This fact can also be obtained by using general formulas for the Weil representations (see for example,  \cite{{sch}, {str}}). Using the transformation property given by \eqref{theta-trans} together with the above observations, we see that 
~$(c\tau+d)^{-\frac{1}{2}}(\theta_{2,0}(g\tau), \theta_{2,2}(g\tau))^t$~ 
is a column vector such that each component have ~$\theta_{2,0}(\tau)$.
Since ~$\theta_{2,0} \rightarrow 1$~ and all other theta components tend to ~$0$~ as
Im$(\tau) \rightarrow \infty$, each component of the above column vector tends to a non-zero limit as Im$(\tau)$ tends to $\infty$. Together with the holomorphicity of ~$h_{2,1}$,~ this shows that 
$(c\tau + d)^{-k+1}(\varphi_0(g\tau),\varphi_2(g\tau))$ tends to a finite limit as Im$(\tau)$ goes to $\infty$. The required cusp conditions now follow. 
\end{proof}

\smallskip

Conversely, let $(\varphi_0, \varphi_2)^t$ be a  vector valued modular form in 
$M_{k-1}(\Gamma_0(2N),\chi;\rho_2)$. We now define $\phi(\tau,z)$ by 
\begin{equation}\label{eq:30}
\begin{split}
\phi(\tau,z) &= \varphi_0(\tau)\theta_{2,1}(\tau)
\theta_{2,0}^J(\tau,z)-{\frac{1}{2}}
(\varphi_0\theta_{2,0}+\varphi_2\theta_{2,2})(\tau)
(\theta_{2,1}^J+\theta_{2,3}^J)(\tau,z) \\
&\hskip 7cm + \varphi_2(\tau)\theta_{2,1}(\tau)\theta_{2,2}^J(\tau,z).
\end{split}
\end{equation}
Using the transformation properties for $(\varphi_0,\varphi_2)$ and the theta functions with respect to   $\Gamma_0(2N)$, we see that $\phi \in J_{k,2}(\Gamma_0(2N), \chi)$. Clearly, by definition, 
$\phi \in {\rm Ker}(D_0)$. Thus, we have obtained the following theorem. 
\begin{theorem}\label{thm:3}
There is a linear isomorphism
$$
\Lambda_2 : J_{k,2}(\Gamma_0(2N),\chi)^0 \longrightarrow  
M_{k-1}(\Gamma_0(2N),\chi;\rho_2), 
$$ 
given by ~$\phi \mapsto (\varphi_0,\varphi_2)^t$, where ~$\varphi_0$~ and ~$\varphi_2$~ are defined by \eqref{eq:22}.
The inverse of ~$\Lambda_2$~ is given by \eqref{eq:30}.
\end{theorem}

\subsection{Connection to the space of cusp forms}

Let $D_2$ be the differential operator as defined in \eqref{eq:15} and $\phi \in {\rm Ker}(D_0)$  be  given by \eqref{eq:30}. Then proceeding as in \cite[Section 3]{AB99} and using 
the differential equations
\begin{equation}\label{eq:16}
{\frac{\partial^2}{\partial z^2}}\theta_{m,r}^J = 4 m (2\pi i){\frac{\partial}
{\partial\tau}}
\theta_{m,r}^J \quad \mbox{ ~for~ } r \in\{0,1,....2m-1\},
\end{equation} 
we obtain
\begin{equation*}\label{eq:31}
D_2(\phi) = 8 k \left( \varphi_0 (\theta_{2,1}\theta_{2,0}' - 
\theta_{2,0}\theta_{2,1}') +
\varphi_2 (\theta_{2,1}\theta_{2,2}' - \theta_{2,2}
\theta_{2,1}') \right).
\end{equation*}
We define $\xi_0 := \theta_{2,1}\theta_{2,0}' - 
\theta_{2,0}\theta_{2,1}'$ and $\xi_2 := 
\theta_{2,1}\theta_{2,2}' - \theta_{2,2}\theta_{2,1}'$. Then
\begin{equation}\label{eq:32}
D_2(\phi) = 8k (\varphi_0, \varphi_2)(\xi_0, \xi_2)^t.
\end{equation}
Proceeding as in \cite[Proposition 2]{AB99}, for any $\gamma=\begin{pmatrix}a&b\\c&d\end{pmatrix}\in \Gamma_0(2)$, we have
\begin{equation}\label{trans-xi}
(\xi_0, \xi_2)^t (\gamma\tau) ~=~ (c\tau+d)^3(\rho_2(\gamma)^{-1})^t (\xi_0, \xi_2)^t (\tau).
\end{equation}
Analysing the behaviour of $\xi_0$ and $\xi_2$ at the cusps of $\Gamma_0(2)$, we find that the vector ~$(\xi_0,\xi_2)^t$~
is a vector valued cusp form for $\Gamma_0(2)$
of weight $3$ and representation 
$({{\rho}_2^{-1})}^t$. Define the space
\begin{equation*}\label{eq:321}
\begin{split}
S_{k+2}(\Gamma_0(2N),\chi)^0 \hskip 12.25cm &\\
~~ := \{ f\in S_{k+2}(\Gamma_0(2N),\chi) {\bf :} ~f =
\varphi_0\xi_0 + \varphi_2\xi_2 {\rm ~with~} (\varphi_0, \varphi_2)^t \in M_{k-1}
(\Gamma_0(2N), \chi;\rho_2)\}.&\\
\end{split}
\end{equation*} 
We summarize the results of \S 3 in the following. 
\begin{theorem}\label{thm:4}
The map $D_2: J_{k,2}(\Gamma_0(2N),\chi) \longrightarrow
S_{k+2}(\Gamma_0(2N),\chi)$ induces an isomorphism between 
$J_{k,2}
(\Gamma_0(2N),\chi)^0$ and $S_{k+2}(\Gamma_0(2N),\chi)^0$.
More precisely, we have the following commutative diagram of isomorphisms: 
$$
\xymatrix{
& {J_{k,2}(\Gamma_0(2N),\chi)^0} \ar[ld]_{\Lambda_2} \ar[dr]^{D_2} \\
{M_{k-1}(\Gamma_0(2N),\chi; {\rho_2})} \ar@{<->}[rr] && 
  S_{k+2}(\Gamma_0(2N),\chi)^0 } 
$$  
where the isomorphism in the bottom of the  diagram is given by\\
$$
(\varphi_0, \varphi_2)^t \mapsto 8 k (\xi_0\varphi_0 + \xi_2\varphi_2).
$$
\end{theorem}

\subsection{Injectivity of $D_0\oplus D_2$} 

In this section, we shall prove the injectivity of the operator $D_0\oplus D_2$ in the weight $2$ case. 
Let $\phi \in {\rm Ker}(D_0\oplus D_2)$. Then $\phi \in {\rm Ker}(D_0)$ and so by \eqref{eq:32} we get 
$D_2(\phi) = 8k (\varphi_0 \xi_0 + \varphi_2 \xi_2)$. Now using the fact that $\phi \in {\rm Ker}(D_2)$, we obtain 
\begin{equation*}
0 = D_2(\phi) = 8k(\varphi_0 \xi_0 + \varphi_2\xi_2),
\end{equation*} 
which gives $\varphi_0\xi_0 + \varphi_2\xi_2 = 0$. Define
\begin{equation}\label{psi}
\psi(\tau) = \frac{\varphi_0}{\xi_2}(\tau) = \frac{-\varphi_2}{\xi_0}(\tau).
\end{equation} 
Using the definitions of ~$\xi_0$~ and ~$\xi_2$, we have
\begin{equation*}
(\theta_{2,2}\xi_0-\theta_{2,0}\xi_2) (\tau) = \theta_{2,1}(\tau)(\theta_{2,2}\theta_{2,0}'-\theta_{2,0}\theta_{2,2}')(\tau) 
= 2\theta_{2,1}\xi^*_2(\tau), 
\end{equation*}
where ~$\xi_2^*$~ is defined by \eqref{xi_m} and is a cusp form of weight $3$ on $\Gamma_0(2)$ with character $\omega_2$.  If for any $\tau \in \mathbb{H}$, $\xi_0(\tau)=\xi_2(\tau) =0$, then the above equation implies (using \eqref{xi})  that 
~$\theta_{2,1}(\tau)\eta^6(2\tau)=0$, which is not true. Therefore the function $\psi$ (defined by \eqref{psi}) is holomorphic in the upper half-plane. 
We shall now prove that $\psi$ is a modular function of weight $k-4$ with respect to the group $\Gamma_0(2N)$. 
Let $\gamma \in \Gamma_0(2N)$. Since $\psi \cdot (\xi_2, -\xi_0)^t = (\varphi_0, \varphi_2)^t$, using the transformation 
\eqref{eq:25} we get the following.  
\begin{equation*}
\psi(\gamma\tau)(\xi_2, -\xi_0)^t(\gamma\tau) = \chi(d) (c\tau+d)^{k-1} \rho_2(\gamma) \psi(\tau)(\xi_2,-\xi_0)^t(\tau).
\end{equation*}
Using \eqref{trans-xi} and following the argument similar to that given in \cite[p. 312]{AB99}, we get the following transformation property 
(and using the definitions of the representation $\rho_2$, characters $r$ and $\omega_2$ given respectively by \eqref{rho2}, \eqref{r-char} 
and \eqref{omega}). 
\begin{equation}\label{trans-psi}
\psi(\gamma \tau)  = \chi(d) \frac{\overline{\omega}_2(\gamma)}{(u_{11}+u_{13})^2} (c\tau +d)^{k-4} \psi(\tau).
\end{equation}

\noindent 
Now, let $k=2$ and $\chi=1$. Consider the function $\psi \xi (\xi_2^*)^3$, which is a weight $10$ cusp form  
on $\Gamma_0(2N)$ and is divisible by $\eta^{18}(2\tau)$, which follows from \eqref{xi}.  
Therefore, for $N$ odd square-free, by using Corollary 2.3 of \cite{AB2003}, we get $\psi \xi (\xi_2^*)^3 =0$, which implies that $\psi=0$. 
When $N=2$, we consider the function $\psi \xi \xi_2^*$ which is a weight $4$ cusp form on $\Gamma_0(4)$. 
Since $S_4(\Gamma_0(4))=\{ 0\}$, in this case also we get $\psi =0$. 
In other words, $\varphi_0 = \varphi_2 =0$ and hence \eqref{eq:30} implies that $\phi =0$. 
Thus, we have the following theorem. 
\begin{theorem}
For $N =2$ or an odd square-free positive integer, the differential map $D_0\oplus D_2$ is injective 
on $J_{2,2}(\Gamma_0(2N))$. 
\end{theorem}

\section{A certain subspace of the space of Jacobi forms of square-free 
index}\label{sect:2}

Throughout this section we assume that $m$ is a square-free positive integer and $N$ is any 
positive integer. Consider the following subspace of Jacobi forms of index $m$ on $\Gamma_0(mN)$:
$$
J_{k,m}^*(\Gamma_{0}(mN), \chi) := \{\phi \in J_{k,m}(\Gamma_0(mN),\chi)~:~ 
h_{m,r}=0 ~~\mbox{for all}~~ r\neq 0,m\}.
$$
When $m=1$, we have  ~$J_{k,1}^*(\Gamma_0(N),\chi)$ ~=~ $J_{k,1}(\Gamma_0(N),\chi)$.
In the case $m=2$, we relate the subspace ~$J_{k,2}^*(\Gamma_0(2N),\chi)$~ 
with the space ~$J_{k,2}(\Gamma_0(2N),\chi)^0$~ in \S 5. Denote the intersection of the space 
~$J_{k,m}(\Gamma_0(mN),\chi)^0$~ with the subspace ~$J_{k,m}^*(\Gamma_{0}(mN), \chi)$~
by $J_{k,m}^*(\Gamma_{0}(mN), \chi)^{0}$. In this section we study the space ~$J_{k,m}^*(\Gamma_{0}(mN), \chi)^{0}$~
and relate this  to the space ~$J_{k,1}(\Gamma_0(N),\chi)^0$, which was studied by Arakawa and 
B\"ocherer \cite{{AB99}, {AB2003}}. 

\subsection{Connection to the space of modular forms}\label{subsect:1}

Suppose $\phi \in J_{k,m}^*(\Gamma_0(mN),\chi)^0$. Then
$0=\phi(\tau,0)=h_{m,0}(\tau)\theta_{m,0}(\tau) + h_{m,m}(\tau)\theta_{m,m}
(\tau)$. We define a new function by 
\begin{equation}\label{eq:11}
\varphi := \frac{h_{m,0}}{\theta_{m,m}} = \frac{-h_{m,m}}{\theta_{m,0}}.
\end{equation}
Since ~$\theta_{m,0}(\tau)=\theta_{1,0}(m\tau)$~ and ~$\theta_{m,m}(\tau) = \theta_{1,1}(m\tau)$, it follows that 
~$\theta_{m,0}$~ and ~$\theta_{m,m}$~ have no zeros in the upper half-plane.
Therefore ~$\varphi$~ defines a holomorphic function in the upper half-plane. 
For any $\gamma = \begin{pmatrix}a&b\\c&d\end{pmatrix}\in \Gamma_0(mN)$, let $\gamma_m=\begin{pmatrix}a&mb\\c/m&d\end{pmatrix}$ (as in \S 2). 
Then using the transformation \eqref{eq:10} and \eqref{eq:11}, we get the following.  
\begin{equation}\label{eq:12}
\varphi(\gamma\tau)(\theta_{m,m}, -\theta_{m,0})^t(\gamma\tau) = \chi(d)
(c\tau+d)^{k-\frac{1}{2}} {\overline{U_1}(\gamma_m)}\varphi(\tau)(\theta_{m,m}, -\theta_{m,0})^t(\tau)
\end{equation}
We now proceed as in the proof of \cite[Proposition 1]{AB99}.  Using \eqref{eq:7} with $z=0$ in  \eqref{eq:12} gives the following transformation property. For any $\gamma = \begin{pmatrix}a&b\\c&d\end{pmatrix}\in \Gamma_0(mN)$,   
\begin{equation}\label{eq:13}
\varphi(\gamma\tau) = \chi(d){\overline{\omega}_m}(\gamma) 
(c\tau+d)^{k-1} \varphi
(\tau).
\end{equation}

\noindent We now study the behaviour of $\varphi$ at the cusps of $\Gamma_0(mN)$.
We can assume that all the cusps of $\Gamma_0(mN)$ are of the form $\frac{a}{c}$
with gcd$(a,c) = 1$ and $c$ varies over positive divisors of $mN$. 
For such a cusp $s = \frac{a}{c}$, choose a matrix ~$g\in SL_2(\Z)$~ such that
~$g(\infty)=s$. Explicitly, we choose ~$g = \begin{pmatrix}a&b\\c&\frac{m}{\alpha}d'
\end{pmatrix}$, where $\alpha = {\rm{gcd}}(m,c)$ and $b,d'$ are integers such that $a \frac{m}{\alpha}d' - bc = 1$. 
Note that gcd$(a \frac{m}{\alpha}, c) = 1$, since $m$ is 
square-free. Let $d=\frac{m}{\alpha}d'$. We have
\begin{eqnarray*}
(c\tau+d)^{-k+\frac{1}{2}}(h_{m,0}(g \tau), h_{m,m}(g\tau)) &=& (c\tau +d)^{-k+
\frac{1}{2}}\varphi(g\tau)
(\theta_{m,m}(g\tau), -\theta_{m,0}(g\tau))\\
&=& (c\tau + d)^{-k+\frac{1}{2}}\varphi(g\tau)
(\theta_{\alpha,0}
(\frac{m}{\alpha}g\tau), \theta_{\alpha,\alpha}(\frac{m}{\alpha}g\tau))~S.
\end{eqnarray*}
Since $\begin{pmatrix}\frac{m}{\alpha}&0\\0&1\end{pmatrix} g \tau$ ~=~ 
$g'\tau'$, where $g' =\begin{pmatrix}a\frac{m}{\alpha}&b\\c&d'\end{pmatrix} \in \Gamma_0(\alpha)$ 
and $\tau' = \frac{\alpha}{m}\tau$. Using \eqref{eq:7}, we get the following equation.
\begin{equation*}
(c\tau+d)^{-k+\frac{1}{2}}(h_{m,0}(g \tau), h_{m,m}(g\tau)) = (c\tau +d)^{-k+1}\varphi(g\tau)
(\theta_{\alpha,0}(\frac{\alpha}{m}\tau), \theta_{\alpha,\alpha}(\frac{\alpha}{m}\tau))
{U_1(g'_\alpha)}^t ~S,
\end{equation*} 
where $g'_\alpha = \begin{pmatrix} am/\alpha & b\alpha\\ c/\alpha & d'\end{pmatrix}$. 
Now, by a similar argument as in the proof of \cite[Proposition 1]{AB99}, we see that $\varphi$ is 
holomorphic at all the cusps of $\Gamma_0(mN)$. 
This shows that  the function $\varphi$ is a modular form of weight $k-1$ for $\Gamma_0(mN)$ with character
$\chi \overline{\omega}_m$. We denote the space of all such modular forms by $M_{k-1}(\Gamma_0(mN),\chi 
\overline{\omega}_m)$.

\smallskip

Conversely, starting with a modular form  $\varphi \in M_{k-1}(\Gamma_0(mN),\chi 
\overline{\omega}_m)$, we obtain a Jacobi form
\begin{equation}\label{eq:14}
\phi(\tau,z) = \varphi(\tau)(\theta_{m,m}(\tau)\theta_{m,0}^J(\tau,z) - 
\theta_{m,0}(\tau)\theta_{m,m}^J(\tau,z)),
\end{equation}
which belongs to $J_{k,m}^*(\Gamma_0(mN), \chi)^0$. We summarize the result of 
this subsection in the following theorem.
\begin{theorem}\label{thm:1}
There is a linear isomorphism
$$
\Lambda_m^* : J_{k,m}^*(\Gamma_{0}(mN), \chi)^0  \longrightarrow  M_{k-1}(\Gamma_0(mN), \chi{\overline{\omega}_m})
$$
given by ~$\phi \mapsto \varphi$, where ~$\varphi$~ is defined by \eqref{eq:11}. The inverse of
~$\Lambda_m^*$~ is given by \eqref{eq:14}.
\end{theorem}

\subsection{Connection to the space of cusp forms}\label{subsect;2}

Let $\phi\in J_{k,2}^*(\Gamma_0(mN),\chi)^0$ be of the form given by 
\eqref{eq:14}.
By
applying $D_2$ and using the differential equations given by 
\eqref{eq:16}, 
 we have
\begin{equation}\label{D-xi}
D_2(\phi)(\tau) = 4m k \varphi(\tau)
(\theta_{m,m}\theta_{m,0}' - \theta_{m,0}
\theta_{m,m}')(\tau)
= 4m^2k \varphi(\tau) \xi_m^*(\tau),
\end{equation}
where $\xi_m^*(\tau)$ is the cusp form defined by \eqref{xi_m}. Now define the space
\begin{equation*}\label{eq:17} 
S_{k+2}^*(\Gamma_0(mN),\chi)^0 := \{ f\in S_{k+2}(\Gamma_0(mN),\chi)~:~ 
f/\xi^*_m 
\in M_{k-1}(\Gamma_0(mN), \chi\overline{\omega}_m)\}.
\end{equation*}
We summarize the results of \S 4.1 and \S 4.2 in the following theorem. 
\begin{theorem}\label{thm:2}
The map $D_2 : J_{k,m}(\Gamma_{0}(mN), \chi) \longrightarrow S_{k+2}
(\Gamma_{0}(mN), \chi)$ induces an isomorphism between 
$J_{k,m}^*(\Gamma_0(mN),
\chi)^0$ and $S_{k+2}^*(\Gamma_0(mN),\chi)^0$. Combining this with 
\linebreak 
Theorem {\rm \ref{thm:1}},  
we get the following commutative diagram of isomorphisms:
$$
\xymatrix{
 & {J_{k,m}^*(\Gamma_0(mN),\chi)^0} \ar[ld]_{\Lambda_m^*} \ar[dr]^{D_2} \\
  {M_{k-1}(\Gamma_0(mN),\chi \overline{\omega}_m)} \ar@{<->}[rr] && 
  S_{k+2}^*(\Gamma_0(mN),\chi)^0 } 
$$            
where the isomorphism in the bottom is given by $\varphi \mapsto 4m^2k 
\xi^*_m\varphi$.
\end{theorem}

\subsection{Connection to the space ~$J_{k,1}(\Gamma_0(N),\chi)^0$}
\label{subsect:3}
Following \cite[Theorem 3.4]{e-z} for the congruence subgroup ~$\Gamma_0(N)$~, we see that the operator $D_0\oplus D_2$ 
is injective
on $J_{k,1}(\Gamma_0(N))$ for all positive even integers $k$. Also by \cite[Theorem 4.3]{AB2003}, $D_0$ is injective on 
$J_{2,1}(\Gamma_0(N))$ for square-free $N$. Now we deduce similar kind of
results for the space $J_{k,m}^*(\Gamma_0(mN))$ in the following two corollaries.
\begin{corollary}\label{cor:1}
The differential map $D_0\oplus D_2 : J_{k,m}^*(\Gamma_0(mN))
\longrightarrow M_k(\Gamma_0(mN))\oplus S_{k+2}(\Gamma_0(mN))$ is injective.
\end{corollary}
\begin{proof}
It is observed in \S 2 that $\xi_m^*$~ has no zeros in the upper half-plane (see \eqref{xi}). 
The corollary follows by using this fact along with \eqref{D-xi}.
\end{proof}

\begin{corollary}\label{cor:2}
The restriction map $D_0 : J_{2,m}^*(\Gamma_0(mN))
\longrightarrow M_2(\Gamma_0(mN))$ is injective, when $mN$ is square-free,
i.e.,  the kernel space $J_{2,m}^*(\Gamma_0(mN))^0 = \{0\}$.
\end{corollary}

\begin{proof}
By Theorem \ref{thm:2}, we see that the spaces $J_{2,m}^*(\Gamma_0(mN))^0$ 
and $S_4^*(\Gamma_0(mN))^0$ are isomorphic, where $S_4^*(\Gamma_0(mN))^0$ is the 
subspace of $S_4(\Gamma_0(mN))$ whose functions are divisible by 
$\xi_m^*$, in other words, divisible by $\eta^6(m\tau)$. By applying the operator ~
$\tau \mapsto -\frac{1}{m\tau}$, the subspace $S_4^*(\Gamma_0(mN))^0$ is equal to the subspace 
whose functions are divisible by $\eta^6(\tau)$. Since ~$mN$~ is square-free,
it follows from \cite[Corollary 
2.3, (2.4)]{AB2003}) that ~$S_4^*(\Gamma_0(mN))^0 = \{0\}$. 
\end{proof}

\smallskip

Let $J_{k,m}(\Gamma_0(mN),\nu)$ denote the space of Jacobi forms
of index $m$ on $\Gamma_0(mN)$ with character $\nu$. 

\begin{corollary}\label{cor:3}
The two kernel spaces  
$J_{k,1}(\Gamma_0(mN),\chi)^0$ and
$J_{k,m}^*(\Gamma_0(mN),\chi\omega_m
\overline{\omega})^0$ are isomorphic.
\end{corollary}
\begin{proof}
Using \cite[Theorem 1]{AB99} and Theorem \ref{thm:1}, both the spaces ~$J_{k,1}(\Gamma_0(mN),\chi)^0$~ and
\linebreak
~$J_{k,m}^*(\Gamma_0(mN),\chi\omega_m
\overline{\omega})^0$~ are isomorphic
to the same space ~$M_{k-1}(\Gamma_0(mN),\chi\overline{\omega})$. This proves the corollary.
\end{proof}

\begin{corollary}\label{cor:4}
Let $N$ be a square-free positive integer and coprime to $m$.
Then $J_{2,m}^*(\Gamma_0(mN),
\omega_m\overline{\omega})^0 = \{0\}$.
\end{corollary}
\begin{proof}
Since $mN$ is square-free, the corollary follows by using Corollary \ref{cor:3} together with \cite[Theorem 4.3]{AB2003}.
\end{proof}

\section{Concluding Remarks}\label{sect:4}

\begin{remark} \label{remark:1}
{\rm 
Note that Theorem \ref{thm:2} reduces to \cite[Theorem 2]{AB99} 
in the case of index 1. Moreover, Corollary \ref{cor:3} shows that there may exist non-trivial examples of isomorphic subspaces in the spaces of Jacobi forms of different index.
}
\end{remark}

\begin{remark} \label{remark:2}
{\rm 
The space $J^*_{k,m}(\Gamma_0(mN), \chi)$ as defined in the section 
\ref{sect:2} can be 
quite large for some values of ~$m$. For example, if $k$ is even and $m$ is 
square-free, 
it is easy to verify that dim $J^*_{k,m}(\Gamma_0(mN)) \geqslant$ dim 
$J_{k,m}(SL_2(\mathbb{Z}))$. Take any Jacobi form $\phi = 
\sum_{j=0}^{2m-1}h_{m,j}\theta^J_{m,j} $ of even weight $k$ and 
square-free index $m$ for the full Jacobi
group $SL_2(\mathbb{Z}) \ltimes \mathbb{Z}^2$. Using \eqref{eq:7}
and \eqref{eq:10},
the function defined by $\psi_{0,m} := h_{m,0}\theta^J_{m,0} 
+ h_{m,m}\theta^J_{m,m}$ is a Jacobi form of weight $k$ and index $m$ for the 
Jacobi group $\Gamma_0(m) \ltimes \mathbb{Z}^2$. Define a mapping 
$ J_{k,m}(SL_2(\mathbb{Z})) \longrightarrow J^*_{k,m}(\Gamma_0(m))$ by 
$\phi(\tau,z) \longmapsto \psi_{0,m}(\tau,z)$. By \cite[Theorem 1]{S05}
we know that this map is injective. Hence, for any positive integer $N$
we have the inequalities ~dim $ J^*_{k,m}(\Gamma_0(mN)) \geqslant$ dim $ J^*_{k,m}(\Gamma_0(m))
\geqslant$ dim $ J_{k,m}(SL_2(\mathbb{Z}))$.
}
\end{remark}

\begin{remark} \label{remark:3}
{\rm 
Suppose $N$ is either $2$ or an odd square-free positive integer.
By using \eqref{eq:10}, it is easy to see that the space $J^*_{k,2}(\Gamma_0(2N),\chi)$ is isomorphic
to $M_{k-\frac{1}{2}}(\Gamma_0(2N),\chi; U^*_2)$, the space of vector valued
modular forms of weight $k-\frac{1}{2}$ and (projective) representation 
$U^*_2$ defined on $\Gamma_0(2)$ by $U^*_2(\gamma) = \overline{U_1}(\gamma_2)$.
It is also not hard to verify that the map
~$(\varphi_0, \varphi_2)^t \longmapsto 
(\frac{\varphi_0}{\theta_{2,1}},\frac{\varphi_2}{\theta_{2,1}})^t$ is an isomorphism from the space
~$M_{k-\frac{1}{2}}(\Gamma_0(2N),\chi; U^*_2)$ onto the space ~$M_{k-1}(\Gamma_0(2N),\chi; \rho_2)$. 
Combining this observation with Theorem \ref{thm:3}, we see that the  
spaces $J_{k,2}(\Gamma_0(2N),\chi)^0$ and 
$J^*_{k,2}(\Gamma_0(2N),\chi)$ are isomorphic.
}
\end{remark}

\begin{remark} \label{remark:4}
{\rm 
Suppose that ~$N$~ is either $2$ or an odd square-free positive integer.
If we consider the space $J^*_{k,2}(\Gamma_0(2N),\chi)^0$ 
as a subspace of the full kernel space $J_{k,2}(\Gamma_0(2N),\chi)^0$,  
then under the isomorphism diagram of Theorem \ref{thm:4} 
any $\phi \in J^*_{k,2}(\Gamma_0(2N),\chi)^0$ will 
correspond to a vector valued modular form ~$(\varphi_0, \varphi_2)^t$~ 
which satisfies the property ~$\varphi_0
\theta_{2,0} + \varphi_2\theta_{2,2} = 0$. Moreover, the image of the space
$J^*_{k,2}(\Gamma_0(2N),\chi)^0$
under $\Lambda_2$, that is, the space ~$\Lambda_2(J^*_{k,2}(\Gamma_0(2N),\chi)^0)$~ is isomorphic to the space 
$M_{k-1}(\Gamma_0(2N), \chi\overline{\omega}_2)$ and the isomorphism is given 
by 
$(\varphi_0, \varphi_2)^t \longmapsto 
\frac{\varphi_0\theta_{2,1}}{\theta_{2,2}} (= \frac{-\varphi_2\theta_{2,1}}
{\theta_{2,0}})$.
}
\end{remark}
 
\smallskip

\noindent {\bf Acknowledgements:} Our heartfelt thanks to Professor S. B\"ocherer for providing many valuable 
suggestions and also for sharing his views on this work. We also thank the referee for his/her meticulous reading 
of the manuscript and suggesting corrections which improved the presentation.

\end{document}